\documentclass[12pt]{amsart}
\usepackage{amssymb, amsmath, amsthm, mathrsfs, amsopn, latexsym}
\usepackage{amsrefs}

\usepackage{graphicx}
\usepackage{color}
\usepackage[a4paper, centering]{geometry}

\usepackage[T1]{fontenc}
\usepackage[utf8]{inputenc}

\newtheorem{definition}{Definition}[section]

\newtheorem{lemma}[definition]{Lemma}

\newtheorem{theorem}[definition]{Theorem}

\theoremstyle{remark}
\newtheorem{remark}[definition]{Remark}

\makeatletter
\def\@settitle{\begin{center}%
		\baselineskip14\p@\relax
		\bfseries
		\uppercasenonmath\@title
		\@title
		\ifx\@subtitle\@empty\else
		\\[5ex]
		\normalsize\mdseries\@subtitle
		\fi
	\end{center}%
}
\def\subtitle#1{\gdef\@subtitle{#1}}
\def\@subtitle{}
\makeatother

\def\R{\mathbb{R}}

\newcommand{\adist}[1][\Omega,K]{\delta_{#1}}

\newcommand{\inrad}{r}
\newcommand{\inradius}[1][\Omega, K]{\inrad_{#1}}
\newcommand{\pscal}[2]{\langle #1,\, #2\rangle}
\newcommand{\Per}[2][K]{\mathcal{P}_{#1}(#2)}
\DeclareMathOperator{\iso}{\mathcal{I}}
\newcommand{\mixed}[1]{V_{(#1)}}
\newcommand{\lstar}{\lambda^*}

\begin{document}
	\title[Inner parallel bodies]
	{
		Anisotropic perimeter\\ 
		and isoperimetric quotient\\
		of inner parallel bodies}
	
	\author[G.~Crasta]{Graziano Crasta}
	\address{Dipartimento di Matematica ``G.\ Castelnuovo'', Univ.\ di Roma I\\
		P.le A.\ Moro 2 -- I-00185 Roma (Italy)}
	\email{crasta@mat.uniroma1.it}

	\date{January 8, 2021}

\begin{abstract}
The aim of this note is twofold:
to give a short proof of the results
in 
[S.\ Larson, \textsl{A bound for the perimeter of inner parallel bodies}, 
J.\ Funct.\ Anal.\ 271 (2016), 610–619]
and
[G.\ Domokos and Z.\ L\'angi, 
\textsl{The isoperimetric quotient of a convex body decreases monotonically
under the eikonal abrasion model}, Mathematika 65 (2019), 119–129];
and to generalize them to the anisotropic case. 
\end{abstract}

\maketitle

\section{Introduction}

Let $\Omega, K\subset\R^n$ be two convex bodies (i.e., compact convex sets) 
with non-empty interior, and let
\[
\Omega \sim \lambda K 
:= \{x\in\R^n\colon x + \lambda K \subset\Omega\}
\qquad \lambda\geq 0,
\]
be the family of \textsl{inner parallel sets} of $\Omega$
relative to $K$,
where $A\sim C := \bigcap_{x\in C} (A - x)$ 
denotes the \textsl{Minkowski difference}
of two convex bodies $A$ and $C$
(see \cite[\S 3.1]{Sch2}).
Let
\[
\inradius[\Omega, K] :=
\max\{\lambda\geq 0\colon
\lambda K + x \subset\Omega \ \text{for some}\ x\in\R^n\}
\]
be the \textsl{inradius of $\Omega$ relative to $K$}, that is,
the greatest number $\lambda$ for which $\Omega\sim \lambda K$
is not empty.

For every convex body $C\subset\R^n$, let $\Per{C}$ denote
its \textsl{anisotropic perimeter} relative to $K$, defined by
\begin{equation}
	\label{f:per}
	\Per{C} := \int_{\partial C} h_K(\nu_C)\, d\mathcal{H}^{n-1}\,,
\end{equation}
where 
$
h_K(\xi) := \sup\{\pscal{x}{\xi}\colon x\in K\}
$
is the \textsl{support function} of $K$,
$\nu_C$ denotes the exterior unit normal vector to $C$,
and $\mathcal{H}^{n-1}$ is the $(n-1)$-dimensional Hausdorff measure.
If $C$ is a convex body with non-empty interior,
$\Per{C}$ coincides with the anisotropic Minkowski content
\begin{equation}
	\label{f:derv}
	\left.\frac{d}{dt}
	V_n(C+t \, K) 
	\right|_{t=0}
	=
	\lim_{t \to 0} \frac{V_n(C+t \, K) - V_n(C)}{t}
	\,,
\end{equation}
where $V_n$ denotes the $n$-dimensional volume
(see \cite[Lemma~7.5.3]{Sch2}).
Furthermore, in the Euclidean setting
(i.e., when $K$ is the unit ball $B$ of $\R^n$),
then $\Per[B]{C} =  \mathcal{H}^{n-1}(\partial C)$.

\medskip
The main results of the present note are Theorems~\ref{thm:anis}
and~\ref{thm:ratio} below,
that have been proved in the Euclidean setting
in \cite[Thm.~1.2]{Larson2016}
and \cite[Thm.~1.1]{Domokos} respectively.
We refer the reader to these papers for motivations and applications.

\begin{theorem}
\label{thm:anis}
(i) Let $\Omega, K\subset\R^n$ be two convex bodies with non-empty interior.
Then it holds that
\begin{equation}
	\label{f:ineq2}
	\Per{\Omega\sim \lambda K} \geq
	\left(1 - \frac{\lambda}{\inradius[\Omega, K]}\right)^{n-1}_+ \Per{\Omega},
	\qquad \forall \lambda\geq 0.
\end{equation}
(ii) 
Equality holds in \eqref{f:ineq2} for some $\lstar\in (0, \inradius[\Omega,K])$
if and only if $\Omega$ is homothetic to a tangential body of $K$.
If this is the case equality holds for all $\lambda\geq 0$ and
every parallel set $\Omega\sim\lambda K$ is homothetic to $\Omega$
for every $\lambda\in [0,\inradius)$.
\end{theorem}

(We postpone to Section~\ref{s:s2} the definition of
tangential body.)

\begin{theorem}
\label{thm:ratio}
Let $\Omega, K\subset\R^n$ be two convex bodies with non-empty interior, 
and let
\[
\iso(\lambda) := \frac{V_n(\Omega\sim\lambda K)}{\Per{\Omega \sim \lambda K}^{\frac{n}{n-1}}}\,,
\qquad
\lambda\in [0, \inradius[\Omega, K])
\]
denote the \textsl{anisotropic isoperimetric quotient}
of $\Omega\sim \lambda K$ relative to $K$.

Then, either $\iso{}$ is strictly decreasing on $[0, \inradius[\Omega, K])$,
or there is some value $\lambda^*\in [0, \inradius[\Omega, K])$ 
such that $\iso{}$ is strictly
decreasing on $[0, \lambda^*]$ and constant on $[\lambda^*, \inradius[\Omega, K])$. 
Furthermore, in the latter
case, for any $\lambda\in[\lambda^*, \inradius[\Omega, K])$, 
$\Omega\sim\lambda K$ is homothetic both to $\Omega\sim\lambda^* K$
and to a tangential body of $K$
(more precisely, to an $(n-2)$-tangential body of $K$).
\end{theorem}

Both results 
can be interpreted in terms of the level sets of
the \textsl{anisotropic distance function} from the boundary of $\Omega$,
defined by
\begin{equation}
\label{f:dist}
\adist(x) := \inf\{\rho_K(y-x)\colon y\in\Omega^c\}
,\qquad x\in\Omega
\end{equation}
where $\rho_K(x) := \max\{\lambda\geq 0\colon \lambda x\in K\}$ is the gauge function of $K$ and we assume that $K$ contains $0$
as an interior point
(see \cite{CM6} for a detailed analysis of $\adist$).
Specifically,
since $\rho_K(x) \leq 1$ if and only if $x\in K$,
it is not difficult to check that
$\Omega\sim\lambda K = \{x\in\Omega\colon \adist(x) \geq \lambda\}$. 

We remark that related results in the Euclidean setting are contained
in \cite[\S 3]{C7},
where, in particular, one can find the proof of
\cite[Thm.~1.2]{Larson2016}
(see p.~104 and Lemma~3.7 therein).

\section{Proof of Theorem~\ref{thm:anis}}
\label{s:s2}

In the following we shall use the notations of \cite{Sch2}.
Let $C\subset\R^n$ be a convex body. We say that $x\in\partial C$ is
a regular point of $\partial C$ if $C$ admits a unique support plane at $x$.
Given two convex bodies $C, K\subset\R^n$,
we say that $C$ is a \textsl{tangential body} of $K$ if,
for each regular point $x$ of $\partial C$, 
the support plane of $C$ at $x$ is also a support plane of $K$
(see \cite[\S 2.2]{Sch2}).
From \cite[Thm.~2.2.10]{Sch2} it follows that
$C$ is a tangential body of a ball if and only if it is homothetic to
its \textsl{form body}, defined by
\[
C_* := \bigcup_{\nu \in S}  \{
x\in\R^n\colon \pscal{x}{\nu} \leq 1\}\,,
\]
where $S$ is the set of outward unit normal vectors to $\partial C$
at regular points of $\partial C$.

The definition of $p$-tangential body is more involved.
Since it is not of primary importance for the exposition 
of the paper,
we refer to \cite[\S 2.2]{Sch2}. 
In connection with the statement of Theorem~\ref{thm:ratio}
we limit ourselves to recall that, if $C$ is a $p$-tangential
body of $K$ for some $p\in \{0, \ldots, n-1\}$,
then it is also a tangential body of $K$.

Given the convex bodies $K_1, \ldots, K_n \subset \R^n$, we denote by
$V(K_1, \ldots, K_n)$ their mixed volume (see \cite[\S 5.1]{Sch2}).
Moreover, for every pair $C,K$ of convex bodies we define
\[
\mixed{i}(C, K) := V(\underbrace{C, \ldots, C}_\text{$n-i$ times}, 
\underbrace{K, \ldots, K}_\text{$i$ times})\,,
\qquad
i\in\{0,\ldots,n\}\,.
\]

From now on we shall assume that
$\Omega, K\subset\R^n$ are two convex bodies with non-empty interior.
To simplify the notation, we denote by $\inrad := \inradius[\Omega, K]$
the inradius of $\Omega$ relative to $K$,
and we define the functions
\[
v_i(\lambda) := \mixed{i}(\Omega\sim\lambda K, K),
\qquad
\lambda\in [0,\inrad],
\quad i\in \{0, \ldots, n\}.
\]
We recall that, by \cite[Lemma~7.5.3]{Sch2},
$v_0$ is differentiable and 
\begin{equation}
\label{f:deriv}
v_0'(\lambda) = -n\, v_1(\lambda),
\qquad
\forall \lambda\in [0,\inrad].
\end{equation}

\begin{theorem}
\label{thm:conc}
(i) The functions
\begin{equation}
\label{f:fi}
f_i(\lambda) := 
v_i(\lambda)^{\frac{1}{n-i}},
\qquad
i\in\{0, \ldots, n-1\},
\end{equation}
are concave in $[0, \inrad]$.

(ii) Assume that there exists $\lstar\in [0,\inrad)$ such that,
for $i=0$ or $i=1$,
\begin{equation}
\label{f:filin}
f_i(\lambda) = \frac{\inrad-\lambda}{\inrad-\lstar}\, f_i(\lstar),
\qquad \forall \lambda\in [\lstar, \inrad].
\end{equation}
Then, for every $\lambda\in [\lstar, \inrad)$,
$\Omega\sim\lambda K$
is homothetic both to  $\Omega\sim \lstar K$,
and to a tangential body of $K$.
\end{theorem}

\begin{proof}
(i) The claim is a direct consequence of the
concavity property of the family $\lambda \mapsto \Omega \sim \lambda\, K$
(see \cite[Lemma~3.1.13]{Sch2}) 
and of the Generalized Brunn--Minkowski inequality
(see \cite[Theorem~7.4.5]{Sch2}).

(ii) Since, by \eqref{f:deriv}, $v_0' = -n\, v_1 = -n\, f_1^{n-1}$ and $v_0(r) = 0$,
if \eqref{f:filin} holds for $i=1$ then it holds also for $i=0$.
Hence, it is enough to prove the claim only in the case $i=0$.

Therefore, assume that \eqref{f:filin} holds for $i=0$ and 
let $\lambda\in [\lstar, \inrad)$.
After a translation, 
we can assume that $\inrad K \subseteq \Omega$, 
so that $(\inrad-\lstar) K \subseteq \Omega\sim \lstar K =: \Omega^*$.
Hence
\[
\begin{split}
\frac{\inrad-\lambda}{\inrad-\lstar}\,\Omega^*
& =
\left[\frac{\inrad-\lambda}{\inrad-\lstar}\,\Omega^* + (\lambda-\lstar) K\right]
\sim (\lambda-\lstar) K
\\
& \subseteq
\Omega^*\sim (\lambda-\lstar) K = \Omega\sim\lambda K\,.
\end{split}
\]
On the other hand, \eqref{f:filin} implies that the sets
$\frac{\inrad-\lambda}{\inrad-\lstar}\,\Omega^*$ and $\Omega\sim \lambda K$
have the same volume,
so that they must coincide, and the conclusion follows.
\end{proof}

The proof of Theorem~\ref{thm:anis}(i)
is a direct consequence of Theorem~\ref{thm:conc}(i),
once we recall that
$\Per[K]{C} = n\, \mixed{1}(C, K)$
(see \cite[(5.34)]{Sch2}).
Specifically,
\[
\Per{\Omega\sim \lambda K}^{\frac{1}{n-1}} = n^{\frac{1}{n-1}}\, f_1(\lambda)
\]
is a concave (non-negative) function in $[0, \inrad]$,
so that \eqref{f:ineq2} follows.

\medskip
Let us prove part~(ii).
Assume that equality holds in \eqref{f:ineq2} for some 
$\lambda_0 \in (0,\inrad)$.
By the concavity of $f_1$ it follows that
the equality holds in \eqref{f:ineq2} for every
$\lambda \in [0,\inrad]$.
Hence, the conclusion follows from Theorem~\ref{thm:conc}(ii).

\section{Proof of Theorem~\ref{thm:ratio}}

Using the notation of Section~\ref{s:s2}, we recall that
\[
v(\lambda) := V_n(\Omega\sim \lambda K) = v_0(\lambda),
\quad
p(\lambda) := \Per{\Omega \sim \lambda K} = n\, v_1(\lambda),
\qquad
\lambda\in [0, \inrad].
\]
By \eqref{f:deriv},
$v$ is differentiable everywhere
with $v'(\lambda) = -p(\lambda)$,
whereas $p$ is differentiable almost everywhere and admits left and
right derivatives at every point,
since $p^{\frac{1}{n-1}}$ coincides, up to a constant factor, 
with the concave function $f_1$.

Hence, $\iso$ is right-differentiable at every point of
$[0, \inrad)$,
and a direct computation shows that
its right derivative is given by
\[
\iso'_+(\lambda) =
-p(\lambda)^{-\frac{2n+1}{n-1}} \, \xi(\lambda),
\qquad
\lambda\in [0, \inrad),
\]
where
\begin{equation}
\label{f:xi}
\xi(\lambda) := p(\lambda)^2 + \frac{n}{n-1}\, v(\lambda) p'_+(\lambda).
\end{equation}

The proof of Theorem~\ref{thm:ratio} is then an easy consequence
of the following result.

\begin{lemma}
\label{lem:tan}
The function $\xi$, defined in \eqref{f:xi},
is non-negative and non-increasing in 
$[0, \inrad)$.
Furthermore, if $\xi$ vanishes at some point
$\lstar\in [0, \inrad)$,
then \eqref{f:filin} holds for $i=0$ and $i=1$, 
and, in addition, 
$\Omega\sim \lstar K$ is homothetic to an $(n-2)$-tangential body
of $K$.
\end{lemma}

\begin{proof}
The function $\xi(-\lambda)/n^2$ coincides with the
function $\Delta(\lambda)$ defined in the proof of Theorem~7.6.19 in \cite{Sch2},
where all the stated properties are proved.
\end{proof}

\begin{remark}
In the planar case $n=2$,
Theorem~\ref{thm:ratio} gives the stronger conclusion
that the isoperimetric quotient is strictly decreasing in $[0,\inrad)$ unless $\Omega$ is
homothetic to $K$, in which case it is constant.
Specifically, assume that $\xi(\lstar) = 0$ for some $\lstar \in [0,\inrad)$;
the stated property will follow if we can prove that $\Omega = r K$.
Since the only $0$-tangential body to $K$ is $K$ itself,
from Lemma~\ref{lem:tan} we deduce that,
for every $\lambda\in[\lstar, r)$, 
$\Omega\sim \lambda K$
is homothetic to $K$.
After a translation we can assume that
$\Omega\sim\lstar K = (r - \lstar) K$.
The concavity property of the family of parallel sets
(see \cite[Lemma~3.1.13]{Sch2}),
together with the fact that
$\Omega\sim\lambda K ={(r - \lambda)} K$ for every 
$\lambda\in [\lstar, r]$, 
imply that
\[
(1-t)\Omega \subseteq (1 - t) r K
\qquad
\forall t\in [\lstar/r,1].
\]
For $t=\lstar/r$ we get the inclusion $\Omega\subseteq r K$; on the other hand,
the opposite inclusion $\Omega \supseteq r K$ follows from the definition of inradius.
\end{remark}

\def\cprime{$'$}

\end{document}